\documentclass[11pt]{article}

\usepackage[alphabetic]{amsrefs}
\usepackage{amsmath,amssymb,amsthm,enumerate,amsfonts,dsfont}
\usepackage[all,cmtip]{xy}
\usepackage[applemac]{inputenc}

\def \1{\mathds 1}
\def \al{\alpha}

\def \C{{\mathbb C}}

\def \CD{{\cal D}}

\def \df{\ \begin{array}{c} _{\rm def}\\ ^{\displaystyle =}\end{array}\ }

\def \eps{\varepsilon}

\def \Id{{\rm Id}}

\def \la{\lambda}

\def \N{{\mathbb N}}
\def \ol{\overline}

\def \Pr{\operatorname{Pr}}
\def \Q{{\mathbb Q}}

\def \sm{\smallsetminus}

\def \tr{\operatorname{tr}}

\def \({\left(}
\def \){\right)}

\newcommand{\bigcupdot}{\mathop{\mathaccent\cdot{\bigcup}}}
  
\newcommand{\e}
[1]{\emph{#1}\index{#1}}

\newcommand{\norm}
[1]{\left|\hspace{-1pt}\left|#1\right|\hspace{-1pt}\right|}

\renewcommand{\sp}
[1]{\left\langle #1\right\rangle}

\newcommand{\tto}
[1]{\stackrel{#1}{\longrightarrow}}

\newtheorem{theorem}{Theorem}[subsection]

\newtheorem{lemma}[theorem]{Lemma}

\newtheorem{proposition}[theorem]{Proposition}
\theoremstyle{definition}
\newtheorem{example}[theorem]{Example}
\newtheorem{definition}[theorem]{Definition}

\begin{document}

\pagestyle{myheadings} \markright{TWO APPLICATIONS OF NETS}

\title{Two applications of nets\\ \ \\ \small
Ann. Funct. Anal. 6. no. 3, 176-190 (2015)}
\author{Ralf Beckmann \& Anton Deitmar}
\date{}
\maketitle

{\bf Abstract.}
Two applications of nets are given. The first is an extension of the Bochner integral to arbitrary locally convex spaces, leading to an integration theory of more general vector valued functions then in the classical approach by Gelfand and Pettis.
The second application starts with the observation that an operator on a Hilbert space is trace class if and only if the net of ``principal trace minors'' converges. The notion of a ``determinant class operator'' then is defined as one for which the net of determinantal principal minors converges.
It is shown that for a normal operator $A$ this condition coincides with $1-A$ being trace class.



$$ $$

%\tableofcontents

%\newpage

\section*{Introduction}

In this paper we give two application of nets.
The first gives a characterization of Bochner integrable functions on arbitrary locally compact spaces.
In this we extend the original approach of Bochner to vector valued integrals to the case of nets, i.e., we approximate a given function by a net of simple functions.
It is slightly astonishing that this actually works, as one is trained to think that nets and integration don't go well together, as for instance the theorems of monotone and dominated convergence fail for general nets.
This is propably the reason why this path hasn't been taken earlier, i.e., why the Bochner integral has not been generalized to arbitrary locally convex spaces.

For Banach spaces, vector-valued integrals have been constructed independently by Bochner \cite{Bochner}
and Gelfand-Pettis \cites{Gelfand, Pettis}.
The latter construction uses a convexity and compactness argument, therefore is a non-constructive existence assertion.
It has been generalized to locally convex spaces in \cite{Bourbaki}, see also \cites{Edwards}. 
The original approach of Bochner, however, has not been generalized to arbitrary locally convex spaces.
This is done in the present paper.
This approach actually allows the integration of more functions on more general spaces than the Gelfand-Pettis method.

The second example starts with the curious observation that trace class operators are exactly those for which the net of ``principal trace minors'' converges, and that the limit is the trace of the operator.
This  serves as a starting point to define the counterpart of trace class operators, the \e{determinant class operators}. We find that in the case of a normal operator $A$ we have
$$
A\text{ is determinant class}\quad\Leftrightarrow\quad 1-A\text{ is trace class},
$$
and that the determinant then equals the Fredholm determinant.
For non-normal operators we still have the ``$\Leftarrow$'' direction of this statement.
Whether the other direction holds, is an open question.

\section{Bochner Integral}
A vector valued integral is a vector attached to a function $f:X\to V$ from a measure space $X$ to a topological complex vector space $V$, written $\int_Xf\,d\mu\in V$ with the property that
$$
\al\(\int_Xf\,d\mu\)=\int_X\al(f)\,d\mu,
$$
for every continuous linear functional $\al:V\to \C$.
This property determines the vector $\int_Xf\,d\mu$ uniquely if the space $V$ is locally convex, as follows from the Hahn-Banach Theorem.
For non-locally convex spaces, the notion rarely makes sense, consider for example the case of the space $V=L^p(0,1)$ with $0<p<1$, \cite{Rudin}.
In this case there are non non-zero continuous linear functionals, therefore every vector is an integral for every function.

So from now on we assume the space $V$ to be locally convex.
In that case the topology of $V$ is generated by all continuous seminorms.
In this note we show the existence and continuity of a vector valued integral in many important cases.
This includes the case of a continuous function $f:X\to V$ of compact support, where $X$ is a locally compact space equipped with a Radon measure.
This case was considered in \cite{Bourbaki}, Chap III \S 4.

For Banach spaces, these integrals have been constructed independently by Bochner \cite{Bochner}
and Gelfand-Pettis \cites{Gelfand, Pettis}.
The latter construction uses a convexity and compactness argument, therefore is a non-constructive existence assertion.
It has been generalized to locally convex spaces before \cites{Bourbaki,Edwards,Garrett}. 
The construction of Bochner works by approximating a given function by simple ones as in the very definition of Lebesgue integration.
It therefore seems more natural and properties of this integral are, as a rule, much simpler to prove as in the Gelfand-Pettis case.
As an example of this rule, consider the strong continuity as in Definition \ref{def} (b).
To the surprise of the authors, the believe was widespread, that the Gelfand-Pettis construction extends to more general settings that the one of Bochner.
In this note we show the contrary.
We extend the approach of Bochner to locally convex spaces which satisfy a mild completeness  condition.

\subsection{Integrable functions}
By a \e{topological vector space} over $\C$ we mean a complex vector space with a topology such that addition and scalar multiplication are continuous maps from $V\times V$ respectively $\C\times V$ to $V$.
We follow the convention that insists that the set $\{ 0\}$ be closed.
This implies that $V$ is a Hausdorff space
as can be seen in \cite{Rudin}, 1.6 or in greater generality in \cite{HA2}, Proposition 1.1.6.

The space is called \e{locally convex}, if every point has a neighborhood base consisting of convex open sets.
Let $V$ be a locally convex topological vector space over $\C$ or locally convex space for short.
Let $(X,\mu)$ be a measure space and $f:X\to V$ a measurable function.
We write $V'$ for the \e{continuous dual space} of $V$, i.e., the space of all continuous linear functionals $\al:V\to\C$.

\begin{definition}\label{def}
We say that $f$ is \e{integrable}, if there exists a vector $\int_Xf\,d\mu\in V$ such that
\begin{enumerate}[\rm (a)]
\item For every  $\al\in V'$ one has
$$
\al\(\int_Xf\,d\mu\)=\int_X \al(f)\,d\mu.
$$ 
\item
For every continuous seminorm $p$ on $V$ one has
$$
p\(\int_Xf\,d\mu\)\ \le\ \int_X p(f)\,d\mu\ <\ \infty.
$$
\end{enumerate}
\end{definition}

\begin{lemma}
If $f:X\to V$ is integrable, then so is $T(f)=T\circ f$ for every continuous linear map $T:V\to W$, where $W$ is another locally convex space.
\end{lemma}

\begin{proof}
Let $g=T(f)$.
Define $\int_Xg\,d\mu=T\(\int_Xf\,d\mu\)\in W$.
For $\al\in W'$ one has $\al\circ T\in V'$.
Therefore ,
\begin{eqnarray*}
\al\(\int_Xg\,d\mu\)&=& \al\circ T\(\int_Xf\,d\mu\)=\int_X\al\circ T(f)\,d\mu=\int_X\al(g)\,d\mu.
\end{eqnarray*}
For (b) let $p$ be a continuous seminorm on $W$. Then $p\circ T$ is a continuous seminorm on $V$.
Therefore,
\begin{eqnarray*}
p\(\int_Xg\,d\mu\)&=& p\circ T\(\int_Xf\,d\mu\)\le\int_Xp\circ T(f)\,d\mu=\int_Xp(g)\,d\mu.
\end{eqnarray*}
\end{proof}

\begin{definition}
A measurable function $f:X\to V$ is called \e{integrally bounded}, if one has
$$
\int_Xp(f)\,d\mu\ <\ \infty
$$
for every continuous seminorm $p$.
\end{definition}

\begin{definition}
The function $f$ is called \e{essentially separable}, if
 for each continuous seminorm $p$ there exists a set $N_p\subset X$ of measure zero and a countable set $C_p\subset V$ such that $f(X\sm N_p)\subset \ol{C_p}^{(p)}$, where the closure is the $p$-closure.
\end{definition}

\begin{example}
Suppose that the image $f(X)$ is relatively compact.
Then $f$ is essentially separable, since for given $n\in\N$ there are $x_1(n),\dots x_{k_n}(n)\in f(X)$ such that 
$$
f(X)\subset \bigcup_{j=1}^{k_n}x_j(n)+\frac1nU_p,
$$
where
$$
U_p=\{ v\in V: p(v)<1\}
$$ is the convex balanced open zero neighborhood attached to the semi-norm $p$.
Let $C_p$ be the set of all $x_j(n)$, where $n$ and $j$ vary.
Then $f(X)\subset \ol{C_p}^{(p)}$, so $f$ is essentially separable.
\end{example}

\begin{definition}
The function $f$ is called \e{essentially bounded}, if there exists a set $N\subset X$ of measure zero, such that $f|_{X\sm N}$ is bounded, which means it is bounded in every continuous seminorm.
\end{definition}

\begin{definition}
The space $V$ is called \e{complete} if every Cauchy-net converges. It is called \e{quasi-complete}, if every bounded Cauchy-net converges.
\end{definition}

\begin{example}
The space $\CD(M)=C_c^\infty(M)$ of smooth functions with compact supports on a smooth manifold $M$ is quasi-complete, as well as its dual space $\CD'(M)$ of distributions on $M$.
For this and more examples, see \cites{Bourbaki0,Garrett,Schaefer}.
\end{example}

\begin{theorem}\label{Thm1.1.9}
Let $V$ be a locally convex space and $f:X\to V$ a measurable function from a measure space $(X,\mu)$.
Suppose that $f$ is essentially separable and integrally bounded.
\begin{enumerate}[\rm (a)]
\item If $V$ is complete, then $f$ is integrable.
\item If $V$ is quasi-complete and $f$ is essentially bounded, then $f$ is integrable.
\item 
If $\mu(X)<\infty$ and the closure of the convex hull of $f(X)$ is complete, then $f$ is integrable.
\end{enumerate}
\end{theorem}

Note that in the Gelfand-Pettis approach of the theorem not only fewer functions are allowed (see below), but under point (c) of the theorem one has to require compactness instead of completeness.

\begin{example}
As an example, consider the case when $X$ is a locally compact Hausdorff space and $\mu$ a Radon measure.
The space $V$ is assumed to be quasi-complete, or even weaker, have the property that the closure of the convex hull of a compact set is complete.
Then any compactly supported continuous map $f:X\to V$ is Bochner-integrable.
We thus get a map
$$
\int_X:C_c(X,V)\to V.
$$
This map is continuous when $C_c(X,V)$ is equipped with the usual inductive limit topology.
\end{example}

Note that the account of \cite{Bourbaki} is explicitly bound to the case of this example, whereas the Bochner construction given here allows the integration of more general functions.
Also in \cite{Edwards} one is restricted to Radon measures and can only integrate functions which are continuous of compact support (up to null-functions).

The proof of the theorem will occupy the rest of the section.

\subsection{Bochner-approximable functions}

\begin{definition}
Let $V$ be a locally convex topological vector space over the complex field.
Let $(X,\mu)$ be a measure space.
A \e{simple function} is a function $s:X\to V$ of the form
$$
s=\sum_{j=1}^n\1_{A_j}v_j
$$
for some measurable sets $A_j\subset X$ of finite measure and some $v_j\in V$.
The \e{integral} of the simple function $s$ equals
$$
\int_X s\,d\mu=\sum_{j=1}^n\mu(A_j)v_j\ \in\ V.
$$
A measurable function $f:X\to V$ is called \e{Bochner-approximable}, if there exists a net $(s_j)_{j\in J}$ of simple functions such that for every continuous seminorm $p$ on $V$ one has
$$
\int_Xp(f-s_j)\,d\mu\ \to\ 0.
$$
In that case the net $(s_j)_j$ is called an \e{approximating net}.

\end{definition}
\begin{lemma}
[net-free formulation]\label{1.1}
A measurable function $f:X\to V$ is Bochner-approximable if and only if for every continuous seminorm $p$ there exists a simple function $s_p$ such that
$$
\int_Xp(f-s_p)\,d\mu\ <\ 1.
$$
\end{lemma}

\begin{proof}
If an approximating net exists, the condition in the lemma is obvious.
Now suppose that the condition of the lemma is satisfied. The set of all continuous seminorms has a natural partial order.
Note that $p\le q$ is equivalent to $U_p\supset U_q$.
As every zero neighborhood contains a convex balanced zero neighborhood, the set of continuous seminorms is directed.
So the simple functions $(s_p)_p$ form a net, and this net will do the job.
This follows from the fact that for every continuous seminorm $p$ and any $\eps>0$ the function $\frac1\eps p$ is again a continuous seminorm and one has
\begin{align*}
\int_Xp(f-s_{\frac1\eps p})\,d\mu<\eps.
\tag*\qedhere
\end{align*}
\end{proof}

\begin{theorem}
Let $f:X\to V$ be Bochner-approximable.
Then for each approximating net $(s_j)_{j\in J}$, the net of integrals
$\(\int_X s_j\,d\mu\)_j$ is a Cauchy net.
If this net converges for one approximating net, then it converges for every approximating net and the limit is uniquely a determined vector
 $\int_Xf\,d\mu$ in $V$.
 In that case we say that $f$ is \e{Bochner-integrable}.
 If $f$ is Bochner-intgeralbe, then so is $T(f)=T\circ f$ for every continuous linear map $T:V\to W$ into another locally convex space $W$.
One then has
$$
T\(\int_Xf\,d\mu\)=\int_X T(f)\,d\mu.
$$
For every continuous seminorm $p$ on $V$ one has
$$
p\(\int_Xf\,d\mu\)\ \le\ \int_X p(f)\,d\mu.
$$
In the case $V=\C$, a function is Bochner integrable if and only if it is Lebesgue integrable, in which case the two integrals coincide.
\end{theorem}

\begin{proof}
For a continuous seminorm $p$ we have
\begin{eqnarray*}
p\(\int_Xs_i\,d\mu-\int_Xs_j\,d\mu\)
&\le& \int_Xp(s_i-s_j)\,d\mu\\
&\le&\int_Xp(s_i-f)\,d\mu+\int_Xp(f-s_j)\,d\mu.
\end{eqnarray*}
This implies that the net of integrals is a Cauchy net.
By the usual argument, the limit does not depend on the choice of the net.
Let $T$ and $(s_j)$ be as in the theorem.
We claim that  $T\circ s_j=T(s_j)$ is a net of simple functions in $W$ which approximates  $T(f)$.
For this let $q$ be a continuous seminorm on $W$.
As $T$ is continuous, there exists a continuous seminorm $p$ on $V$ such that $q(T(v))\le p(v)$ for every $v\in V$.
We conclude
$$
\int_Xq(T(f)-T(s_j))\,d\mu\le\int_Xp(f-s_j)\,d\mu.
$$
So $T(s_j)$ indeed approximates $T(f)$ and so
\begin{eqnarray*}
T\(\int_Xf\,d\mu\) &=& T\(\lim_j\int_Xs_j\,d\mu\)\\
&=& \lim_j\int_XT(s_j)\,d\mu= \int_XT(f)\,d\mu.
\end{eqnarray*}
The assertion about the case $V=\C$ is easy.
For a continuous seminorm $p$ we have
\begin{eqnarray*}
p\(\int_Xf\,d\mu\) &=& p\(\lim_j\int_Xs_j\,d\mu\)\\
&=& \lim_jp\(\int_Xs_j\,d\mu\)\ \le\  \liminf_j\int_Xp(s_j)\,d\mu.
\end{eqnarray*}
Let $\eps>0$.
There exists $j_0$ such that for every $j\ge j_0$ one has $\int_Xp(s_j-f)\,d\mu<\eps$.
As $|p(s_j)-p(f)|\le p(s_j-f)$ we conclude  $p\(\int_Xf\,d\mu\)<\int_Xp(f)\,d\mu+\eps$.
For $\eps\to 0$ the claim follows.
\end{proof}

\subsection{Integrable functions}

\begin{lemma}\label{2.1}
If $f$ is essentially separable, then for each continuous seminorm $p$, the set $C_p$ can be chosen inside the image $f(X)$.
\end{lemma}

\begin{proof}
Let $c\in\C_p$ and let $n\in\N$.
If $(c+\frac1nU_p)\cap f(X)\ne\emptyset$, we choose an element $y(c,n)$ in that set.
Let $D_p$ be the union of all these elements $y(c,n)$.
We claim that $f(X)\subset\ol{D_P}^{(p)}$.
To prove this, let $x\in X$ and $n\in\N$.
There exists $c\in C_p$ with $p(f(x)-c)<\frac1{2n}$.
So $(c+\frac1{2n}U_p)\cap f(X)\ne\emptyset$, i.e., the element $y(c,2n)\in f(X)$ with $p(y(c,2n)-c)<\frac1{2n}$ exists.
It follows that $p(y(c,n)-f(x))<\frac1n$.
\end{proof}

The next theorem generalizes Theorems 3 and 6 of \cite{Thomas}.

\begin{theorem}\label{2.2}
A measurable function $f:X\to V$ is Bochner-approximable if and only if
\begin{enumerate}[\rm (a)]
\item $f$ is essentially separable and
\item $f$ is integrally bounded.
\end{enumerate}
If $f(X)$ is relatively compact and $\mu(X)<\infty$, then $f$ is Bochner-approximable.
\end{theorem}

\begin{proof}
Let $f$ be Bochner-approximable.
We show that $f$ is essentially separable.
So let $p$ be a continuous seminorm.
For each given $n\in \N$, there exists a simple function $s_n:X\to V$ with $\int_Xp(f-s_n)\,d\mu<\frac1n$.
Let $E_p$ be the $p$-closure of the vector space spanned by the union of the images of all $s_n$, $n\in\N$.
Then $E_p$ is the $p$-closure of some countable set $C_p$, for instance, one can take the $\Q(i)$-vector space spanned by the images of all $s_n$.
For each $n\in\N$ the set
$$
N_n=\left\{ x\in X:p(f(x),E_p)>\frac1n\right\}
$$
is a set of measure zero, where 
$$
p(v,E_p)=\inf\{p(v-e):e\in E_p\}.
$$
The complement in $X$ of the set $f^{-1}(E_p)$ is the union of all $N_n$, therefore a set of measure zero, so $f$ is essentially separable.
As $\int_X p(f-s_j)\,d\mu<\infty$ it follows that $\int_Xp(f)\,d\mu\le\int_Xp(f-s_j)+p(s_j)\,d\mu<\infty$, so $p(f)$ is integrable.

Now for the converse direction.
Let $p$ be a continuous seminorm.
We will attach to $p$ a simple function $s_p$ with $\int_Xp(f-s_p)\,d\mu<1$.
Then $f$ is Bochner-approximable by Lemma \ref{1.1}.
In order to construct $s_p$, let $C_p=\{ c_1,c_2,\dots\}$ be the countable set  and let $N_p\subset X$ be the nullset attached to $p$.
Write $X_p=X\sm N_p$.
For $n\in\N$ and $\delta>0$ let $A_n^\delta$ be the set of all $x\in X_p$ such that $p(f(x))>\delta$ and $p(f(x)-c_n)<\delta$.
To have a sequence of pairwise disjoint sets, define
$$
D_n^\delta= A_m^\delta\sm \bigcup_{k<n}A_k^\delta.
$$
The set $\bigcup_nA_n^\delta=\bigcupdot_nD_n^\delta$ equals $f^{-1}(f(X_p)\sm \delta U_p)$.
Since $p(f)$ is integrable, the set $\bigcupdot D_n^\delta$ is of finite measure.
Let $s_{p,n}=\sum_{j=1}^n\1_{D_j^{1/n}}c_j$.
This is a simple function.
It is easy to see that the sequence $p(s_{p,n}-f)$ converges to $0$ pointwise on the set $X_p$.
On that set we also have $p(s_{p,n})\le 2p(f)$ by construction.
So we get $p(f-s_{p,n})\le p(f)+p(s_{p,n})\le 3p(f)$, and by dominated convergence,
$$
\int_Xp(f-s_{p,n})\,d\mu\ \to\ 0.
$$
In particular, there exists $n_0\in\N$, such that for $s_p=s_{p,n_0}$ one has $\int_Xp(f-s_p)\,d\mu<1$.

Finally, assume that $f(X)$ is relatively compact and $\mu(X)<\infty$. Then for any given continuous seminorm $p$, the set $p(f(X))$ is relatively compact, hence bounded, so $p(f)$ is integrable since $\mu(X)<\infty$.
Further, $f$ is essentially separable by Example 2.1.
\end{proof}

\begin{theorem}\label{2.4}
\renewcommand{\labelenumi}{\rm(\roman{enumi})}
\begin{enumerate}
\item If $V$ is complete, then every Bochner-approximable function in $V$ is Bochner-integrable.
\item
If $V$ is quasi-complete, then every bounded Bochner-approximable function in $V$ is Bochner-integrable.
\item
Let $f:X\to V$ be Bochner-approximable.
If $\mu(X)<\infty$ and the closure of the convex hull of $f(X)$ is complete, then $f$ is Bochner-integrable.
\end{enumerate}
\renewcommand{\labelenumi}{\arabic{enumi}}
\end{theorem}

\begin{proof} (i) and (ii) are clear.
For (iii) we may assume $\mu(X)=1$.
By Lemma \ref{2.1} the set $C_p$ can be chosen inside $f(X)$.
According to the proof of Theorem \ref{2.2} there exists an approximating net $(s_j)_{j\in J}$ such that each $s_j$ takes values in the sets $C_p$ for varying $p$, hence $s_j(X)\subset f(X)$.
Now write
$$
s_j=\sum_{k=1}^n\1_{A_k}v_k,
$$
then each $v_k$ lies in $f(X)$ and $X=\bigcupdot_kA_k$.
Therefore,
$$
\int_Xs_j\,d\mu=\sum_{k=1}^n\mu(A_k)v_k
$$
is a convex-combination of elements of $f(X)$, so lies in the convex hull of $f(X)$.
The closure of this convex hull being complete, the Cauchy-net $s_j$  converges.
\end{proof}

\section{Traces and Determinants}
An operator $T$ on a Hilbert space $H$ is called a \e{trace class operator} if its singular values are summable \cites{Simon,HA2}.
In this section we find that this is equivalent with the net of ``principal trace minors'' being convergent.
We take that as a motivation to then define the notion of a \e{determinant class operator} to be one for which the determinantal principal minors converge.
It turns out that in the normal case these are exactly those opertors $A$, for which $1-A$ is of trace class.
In the non-normal case the correspoinding assertion remains an open question.

\subsection{Trace class}
Let $T:H\to H$ be a bounded operator on a Hilbert space $H$.
For a finite-dimensional subspace $F\subset H$ let $T_F:F\to F$ be given by
$$
T_F:F\hookrightarrow H\tto TH\twoheadrightarrow F,
$$
where the first arrow is the inclusion of $F$ in $H$ and the last one is orthogonal projection $\Pr_F$ onto $F$.

\begin{theorem}\label{thm1.1}
Let $T:H\to H$ be a bounded operator on a Hilbert space $H$.
Then $T$ is trace class if and only if the limit of ``principal trace minors'':
$$
\lim_F\tr(T_F)
$$
exists.
In this case the limit equals the trace of $T$.
\end{theorem}

\begin{proof}
Let $T:H\to H$ be a bounded operator and write $\tau(T)$ for the limit in the theorem.
Setting $A=\frac12(T+T^*)$ and $B=\frac1{2i}(T-T*)$, the operators $A$ and $B$ are bounded self-adjoint and $T=A+iB$. 
Now $T$ is trace class if and only if $A$ and $B$ are. We first show that $\tau(T)$ exists if and only if $\tau(A)$ and $\tau(B)$ exists.
Let $F\subset H$ be a finite-dimensional subspace.
Let $e_1,\dots,e_n$ be an orthonormal basis of $F$, then
$$
\tr(T_F)=\sum_{j=1}^n\sp{Te_j,e_j},
$$
from which we deduce
$$
\tr\((T^*)_F\)=\ol{\tr(T_F)}.
$$
It follows that if $\tau(T)$ exists, then $\tau(T^*)$ exists and equals $\ol{\tau(T)}$.
Further it follows that if $\tau(T)$ exists, then so do $\tau(A)$ and $\tau(B)$ and vice versa.
So, in order to prove the theorem, it suffices to assume that $T$ is self-adjoint.
Then, by the spectral theorem, there exists positive operators $T_+$ and $T_-$, commuting with $T$ and each other such that $T=T_+-T_-$, which means that it suffices to prove the theorem in the case that $T$ is positive.

Assume first that $T$ is trace class and let $\la_1\ge\la_2\ge\dots\ge 0$ be its eigenvalues counted with multiplicities.
For given $\eps>0$ there exists $N\in\N$ such that $\sum_{j=N+1}^\infty\la_j<\eps$.
Let $F_0$ be the span of $v_1,\dots,v_N$ where the $v_j$ are linearly independent and satisfy $Tv_j=\la_jv_j$.
Let $F\supset F_0$ be a finite-dimensional subspace of $H$, then
$$
0\le\tr(T)-\tr(T_F)\le\tr(T)-\tr(T_{F_0})=\sum_{j=N+1}^\infty\la_j<\eps.
$$
This means that $\tr(T_F)$ converges to $\tr(T)$.

Now for the other direction assume that $T\ge 0$ and the limit of $\tr(T_F)$ exists.
Let $\al>0$ and let $W$ be the image of the spectral projection $\mu_T(\al,\infty)$, where $\mu_T$ is the spectral measure of $T$.
This means that $W$ and $W^\perp$ are stable under $T$, that $T$ has norm $\le\al$ on $W^\perp$ and $\sp{Tw,w}\ge \al\sp{w,w}$ holds for every $w\in W$.
The existence of $\lim_F\tr(T_F)$ implies  that $W$ must be finite-dimensional, hence $T$ diagonalizes on $W$ and by letting $\al$ tend to zero we find that $T$ is compact and as the limit of the traces exists, $T$ must be trace class.
\end{proof}

\subsection{Determinant class}
Let $A:H\to H$ be a bounded operator on a Hilbert space $H$.
We say that $A$ is of \e{determinant class}, if the limit of principal minors
$$
\det(A)\df\lim_F\det(A_F)
$$
exists in $\C$ and is $\ne 0$, where the limit is extended over the net $\det(A_F)$ of complex numbers indexed by the directed set of all finite-dimensional subspaces $F$ of $H$.

Let $T$ be a trace class operator and recall the \e{Fredholm determinant}
$$
\det(1-T)=\sum_{k=0}^\infty (-1)^k\tr\(\bigwedge^kT\),
$$
where the sum converges absolutely, moreover, one has
$$
\sum_{k=0}^\infty \norm{\bigwedge^kT}_{\tr}<\infty,
$$
where $\norm._{\tr}$ is the \e{trace norm}, see \cite{Simon}.
One always has $|\tr(T)|\le\norm T_{\tr}=\tr(|T|)$ and $\norm{ST}_{\tr}\le\norm S\norm{T}_{\tr}$, where $S$ is any bounded operator.
If $T$ is normal and $\la_1,\dots$ are the eigenvalues counted with multiplicity, then one has
$$
\det(1-T)=\prod_{j=1}^\infty (1-\la_j)
$$
and the product converges absolutely.
If $S,T$ are trace class, then the fredholm determinant of $(1-S)(1-T)$ exists and 
$$
\det(1-S)(1-T)=\det(1-S)\det(1-T).
$$

\begin{proposition}
Let $T$ be a trace class operator, then $A=1-T$ is of determinant class and the determinant equals the Fredholm determinant.
\end{proposition}

\begin{proof}
Let $F\subset H$ be finite-dimensional and let $P_{k,F}$ be the orthogonal projection from $\bigwedge^kH\to\bigwedge^kF$.
Then 
$$
|\tr\(\bigwedge^kT_F\)|\le\norm{\bigwedge^kT_F}_{\tr}
=\norm{P_{k,F}\bigwedge^kT}_{\tr}\le 
\norm{\bigwedge^kT}_{\tr}.
$$
As the directed set of all spaces of the form $\bigwedge^kF$ is strongly cofinal in the set of all finite dimensional subspaces of $\bigwedge H$, it follows
by Theorem \ref{thm1.1}, that the trace $\tr\(\bigwedge^kT_F\)$ converges to $\tr\(\bigwedge^kT\)$ as $F\to H$.
We now argue that the determinant
$$
\det(1-T_F)=\sum_{k=0}^\infty (-1)^k\tr\(\bigwedge^k T_F\)
$$
converges to $\det(1-T)$ by dominated convergence.
Of course there is no theorem of dominated convergence for nets in general, however, if the measure space is countable, there is. We have formulated it in the next lemma.
The proof of the proposition is finished.
\end{proof}

\begin{lemma}
[Dominated convergence for nets on countable spaces]\label{lemDomConv}
Let $I\ni\al\to a_\al$ be a net of sequences $(a_{\al,k})_{k\in\N}$ of complex numbers.
Assume that the net converges pointwise, i.e.,
$
a_{\al,k}\to a_k
$
as $\al\to\infty$ for some $a_k\in\C$.
Assume further that there exists a sequence $g_k\ge 0$ with $|a_{\al,k}|\le g_k$ for all $\al\in I$ and $\sum_{k=1}^\infty g_k<\infty$.
Then the sum $\sum_{k=1}^\infty a_{\al,k}$ converges to $\sum_{k=1}^\infty a_k$ as $\al\to\infty$.
\end{lemma}

\begin{proof}
For given $\eps>0$ there is $k_0$ such that $\sum_{k=k_0+1}^\infty g_k<\eps/4$.
Next there exists $\al_0\in I$ such that for all $\al\ge \al_0$ one has
$$
|a_{\al,k}-a_k|<\frac\eps{2^{k+1}}
$$
holds for every $1\le k\le k_0$.
For every $\al\ge \al_0$ it follows that
\begin{align*}
\left|\sum_{k=1}^\infty a_{\al,k}-\sum_{k=1}^\infty a_k\right|&\le\sum_{k=1}^{k_0}|a_{\al,k}-a_k|+\sum_{k=k_0+1}^\infty |a_{\al,k}|+|a_k|\\
&\le \sum_{k=1}^{k_0}\frac{\eps}{2^{k+1}}+2\sum_{k=k_0+1}^\infty g_k\\
&<\frac\eps 2+\frac\eps 2=\eps.\tag*\qedhere
\end{align*}
\end{proof}

\begin{lemma}\label{positiveReduct}
Let $A:H\to H$ be of determinant class and assume that $A$ respects an orthogonal decomposition $H=U\oplus U^\perp$ for some closed subspace $U$.
Then $A_U$ and $A_{U^\perp}$ are of determinant class and one has
$$
\det(A)=\det(A_U)\det(A_{U^\perp}).
$$
As a side result we note that if $\eps>0$ and $F\subset H$ is s finite-dimensional subspace such that $|\det(A_{F'})-\det(A)|<\eps|\det(A)|/2$ for every finite dimensional $F'\supset F$, then we have
$$
|\det(A_{F_U})-\det(A_U)|\le \eps |\det(A_U)|.
$$
\end{lemma}

\begin{proof}
Let $\eps>0$ smaller than $|\det(A)|$ and choose a finite-dimensional space $F\subset H$ such that $|\det(A_{F'})-\det(A)|<\eps/2$ for every finite-dimensional $F'\supset F$.
Replacing $F$ by $\Pr_U(F)\oplus\Pr_{U^\perp}(F)$ we may assume that $F=F_U\oplus F_{U^\perp}$, where $F_U=F\cap U$ and $F_{U^\perp}=F\cap U^\perp$.
Let then $F^1_U$ be any finite dimensional subspace of $U$ containing $F_U$.
Setting $F^1=F_U^1\oplus F_{U^\perp}$ we then have
\begin{align*}
|\det(A_{F^1_U})\det(A_{F_{U^\perp}})-\det(A)|=|\det(A_{F^1})-\det(A)|<\eps/2.
\end{align*}
Note that this implies that $|\det(A_{F^1_U})\det(A_{F_{U^\perp}})|>\frac{|\det(A)|}2$.
Since the same holds with $F_U$ instead of $F_U^1$ we get
$$
|\det(A_{F^1_U})\det(A_{F_{U^\perp}})-\det(A_{F_U})\det(A_{F_{U^\perp}})|<\eps
$$
or $|\det(A_{F^1_U})-\det(A_{F_U})|<\eps/|\det(A_{F_{U^\perp}})|$.
So the net $F_U^1\mapsto \det(A_{F^1_U})$ is bounded, hence has a convergent subnet with a non-zero limit which we call $\det(A_U)$.
The same holds true for the net $F_{U^\perp}^2\mapsto \det(A_{F_{U^\perp}}^2)$ which has a subnet converging to a complex number we call $\det(A_{U^\perp})$.
Taking the limit, we get $|\det(A_U)\det(A_{U^\perp})-\det(A)|\le\eps/2$ and as $\eps$ is arbitrary we find that the limits are unique, giving the claim.
\end{proof}

\begin{theorem}
Let $A:H\to H$ be a normal operator of determinant class.
Then $T=1-A$ is of trace class and the determinant $\det(A)=\det(1-T)$ is a Fredholm determinant.
\end{theorem}

\begin{proof}
Let $\mu_A$ be the spectral measure of $A$, so $\mu_A$ is a projection valued measure defined on the Borel $\sigma$-algebra of $\C$, supported on the spectrum of $A$, such that $A=\int_\C t\,d\mu_A(t)$.
Further one has $\mu_A(\C)=\Id_H$. 
We first show that $A$ has discrete eigenvalue spectrum with only accumulation point 1.
For this it suffices to show that for each $z_0\in\C\sm \{1\}$ there exists a closed ball $B=B_r(z_0)$ around $z_0$ of some radius $r>0$ such that $\mu_A(B)H$ is finite-dimensional.
Suppose first that $|z_0|>1$.
Then $B$ can be chosen in a way that for every $z\in B$ one has $|z|\ge\theta$ for some fixed $\theta>1$.
The orthogonal decomposition
$$
H=\mu_A(B)H\oplus\mu_A(\C\sm B)H
$$
is stable under $A$ and by Lemma \ref{positiveReduct} we conclude that $A_W$ is determinant class, where $W=\mu_A(B)H$.
The spectrum of the normal operator $A_W$ is contained in the convex compact set $B$ and we claim that for every finite-dimensional subspace $F\subset W$ the eigenvalues of $A_F$ lie in the set $B$.
For this let $\la$ be an eigenvalue and $v\in F$ an eigenvector of norm one, then
\begin{align*}
\la=\la\sp{v,v}&=\sp{A_Fv,v}=\sp{\Pr_FAv,v}=\sp{Av,\Pr_Fv}\\
&=\sp{Av,v}
=\int_B t \sp{d\mu_A(t)v,v}.
\end{align*}
As $\sp{\mu_A(t)v,v}\in [0,1]$ and $\int_B\sp{d\mu_A(t)v,v}=1$, $\la$ lies in the closed convex hull of $B$ which is $B$ itself.
Therefore, if $F\subset W$ is a subspace of dimension $N$ it follows that $A_F$ has eigenvalues $\la_1,\dots,\la_N$ each of which satisfies $|\la_j|\ge \theta$.
Therefore
$$
|\det(A_F)|=|\la_1\cdots\la_N|\ge \theta^N.
$$
As $N$ increases, this tends to infinity contradicting convergence, so $N$ must be bounded.
The case $|z_0|<1$ is treated similarly.

The case $|z_0|=1$ is more subtle.
In this case assume that the space $W(r)=\mu_A(B_r(z_0))H$ is infinite-dimensional for every $r>0$.
We start with sufficiently small $\eps>0$ and choose a finite-dimensional subspace $F\subset H$ such that $|\det(A_F)-\det(A)|<\eps |\det(A)|/2$.
Let $N=\dim F$ and choose $0<r<1$ so small that $B_r(z_0)^{N}$ is contained in $B_r(z_0^{N})$ for some small $\delta>0$ which we determine later.
Let $U=W(r)$ and set $F_0=\Pr_U(F)$.
We then fix a vector $v_1\in W(r/2)$ linearly independent from $F_0$, next a vector $v_2\in W(r/4)$ linearly independent from the span of $F_0$ and $v_1$ and so on.
In the $n$-th step we choose a vector $v_n\in W(r/2^{n})$ linearly independent from $F_0$ and $v_1,\dots,v_{n-1}$.
Let $F_n$ denote the span of $F_0$ and $v_1,\dots,v_n$.
By Lemma \ref{positiveReduct} we get
$$
|\det(A_{F_n})-\det(A_U)|\le\eps|\det(A_U)|
$$
for every $n$.
Using the assumption, we are now going to contradict this estimate.
Let $\la_1,\dots,\la_{N+n}$ be the eigenvalues of $A_{F_n}$ ordered as follows:
As the space $W(r/2^n)$ is stable under $A$, there must be an eigenvalue $\la_{N+n}$ lying in $B_{r/2^n}(z_0)$.
Likewise, $\la_{N+n-1}$ can be assumed in $B_{r/2^{n-1}}(z_0)$ and so on.
Finally $\la_1,\dots,\la_N\in B_r(z_0)$.
As we have $|\la_1\cdots \la_N-z_0^N|\le\delta$ and $|\la_{N+j}-z_0|\le r/2^j$ and $|z_0|=1$ we get
\begin{align*}
|\la_1\cdots\la_{N+n}-z_0^{N+n}|
&\le |\la_1\cdots\la_{N+n}-z_0^N\la_{N+1}\cdots\la_{N+n}|\\
&\ \ +
|\la_{N+1}\cdots\la_{N+n}-z_0^{n}|\\
&\le \delta |\la_{N+1}\cdots\la_{N+n}|\\
&\ \ +|\la_{N+1}\cdots\la_{N+n}-z_0\la_{N+2}\cdots\la_{N+n}|\\
&\ \ +|\la_{N+2}\cdots\la_{N+n}-z_0^{n-1}|\\
&\le \delta |\la_{N+1}\cdots\la_{N+n}|\\
&\ \ +\frac r2|\la_{N+2}\cdots\la_{N+n}|\\
&\ \ +|\la_{N+2}\cdots\la_{N+n}-z_0^{n-1}|
\end{align*}
This iterates to
\begin{align*}
|\la_1\cdots\la_{N+n}-z_0^{N+n}|
&\le \delta |\la_{N+1}\cdots\la_{N+n}|\\
&\ \ +\frac r2|\la_{N+2}\cdots\la_{N+n}|\\
&\ \ +\frac r4|\la_{N+3}\cdots\la_{N+n}|\\
&\ \ \vdots\\
&\ \ +\frac r{2^n}.
\end{align*}
Now $|\la_{N+1}\cdots\la_{N+n}|\le (1+\frac12)(1+\frac14)\cdots$.
So setting $C=\prod_{j=1}^\infty(1+\frac1{2^j})$ we have
\begin{align*}
|\det(A_{F_n})-z_0^{N+n}|&=|\la_1\cdots\la_{N+n}-z_0^{N+n}|\\
&\le \delta C+C\sum_{j=1}^\infty \frac1{2^j}\\
&=(\delta+r)C.
\end{align*}
Putting things together we arrive at
$$
|\det(A_U)-z_0^{N+n}|\le\eps|\det(A_U)|+(\delta+r)C.
$$
Choosing $\eps$ and $\delta$ small, we arrive at a contradiction.

Let $\sigma(A)$ denote the spectrum of the normal operator $A$. We deduce that $\sigma(A)\sm\{1\}$ consists of eigenvalues only, which can only accumulate at $1$.
Let $F\subset H$ be a finite-dimensional subspace such that $|\det(A_{F'})-\det(A)|<|\det(A)|/2$ for every finite dimensional subspace $F'$ with $F'\supset F$.
Let $v_1,v_2,\dots$ be a maximal family of normalized orthogonal eigenvectors for $A$ such that the span of $v_1,v_2,\dots$ has zero intersection with $F$.
As $F$ is finite-dimensional, only finitely many eigenvalues of $A$ are left out.
For each $n\in\N$ denote by $F_n$ the span of $F$ together with $v_1,\dots,v_n$.
Then
$$
|\la_1\cdots\la_n\det(A_F)-\det(A)|<|\det(A)|/2
$$
for every $n$, or
$$
|\la_1\cdots\la_n-\frac{\det(A)}{\det(A_F)}|<\frac{|\det(A)|}{2|\det(A_F)|}.
$$
So the sequence $\la_1\cdots\la_n$ remains bounded.
As the sequence $\la_j$ tends to 1 and by increasing $F$ we can assume that $\frac{det(A)}{\det(A_F)}$ is as close to 1 as we want, we can assume that
$$
\log(\la_1\cdots\la_n)=\sum_{j=1}^n\log(\la_j),
$$
and the latter remains bounded for all $n$, even if we change the order of the $\la_j$.
This, however implies that the real part of the series remains bounded by a bound independent of reordering, so it must converge absolutely.
The same holde for the imaginary part and so the series $\sum_{n=1}^\infty\log(\la_j)$ converges absolutely, which implies the theorem. 
\end{proof}

\begin{proposition}
Let $A,B$ be normal determinant class operators.
Then $AB$ is of determinant class and 
$$
\det(AB)=\det(A)\det(B).
$$
\end{proposition}

\begin{proof}
By the theorem we can write $A=1-S$ and $B=1-T$ with trace class operators $S,T$.
Then $AB=(1-S)(1-T)=1-S-T+ST$ and $S+T-ST$ is of trace class again.
The equation $\det(AB)=\det(A)\det(B)$ follows form the properties of the Fredholm determinant \cite{Simon}.
\end{proof}

{\bf Questions.}
\begin{itemize}
\item Let $A$ be of determinant class. Does it follw that $1-A$ is of trace class?
\item Suppose that $A,B$ are determinant class.
Does it follow that $AB$ and $A\oplus B$ are determinant class?
\end{itemize}

{\small 
- Jablonskistr. 6A,
10405 Berlin\\
- Mathematisches Institut,
Auf der Morgenstelle 10,
72076 T\"ubingen,
Germany\\
\tt \phantom 1 deitmar@uni-tuebingen.de}

\end{document}